\documentclass[11pt]{article}
\usepackage{amssymb}
\newtheorem{precor}{{\bf Corollary}}

\newenvironment{cor}{\begin{precor}{\hspace{-0.5
               em}{\bf.\ }}}{\end{precor}}
\newtheorem{precon}{{\bf Conjecture}}

\newtheorem{prealphcon}{{\bf Conjecture}}

\newenvironment{alphcon}{\begin{prealphcon}{\hspace{-0.5
               em}{\bf.\ }}}{\end{prealphcon}}
\newtheorem{predefin}{{\bf Definition}}

\newtheorem{preexm}{{\bf Example}}

\newtheorem{preappl}{{\bf Application}}

\newtheorem{prelem}{{\bf Lemma}}

\newenvironment{lem}{\begin{prelem}{\hspace{-0.5
               em}{\bf.\ }}}{\end{prelem}}
\newtheorem{preproof}{{\bf Proof.\ }}

\newenvironment{proof}[1]{\begin{preproof}{\rm
               #1}\hfill{$\blacksquare$}}{\end{preproof}}
\newtheorem{prethm}{{\bf Theorem}}

\newenvironment{thm}{\begin{prethm}{\hspace{-0.5
               em}{\bf.\ }}}{\end{prethm}}
\newtheorem{prealphthm}{{\bf Theorem}}

\newenvironment{alphthm}{\begin{prealphthm}{\hspace{-0.5
               em}{\bf.\ }}}{\end{prealphthm}}

\newtheorem{prealphlem}{{\bf Lemma}}

\newtheorem{prepro}{{\bf Proposition}}

\newtheorem{preprb}{{\bf Problem}}

\newtheorem{preapp}{{\bf Application}}

\newtheorem{prequ}{{\bf Question}}

\def\conct[#1,#2]{\mbox {${#1} \leftrightarrow {#2}$}}
\def\dconct[#1,#2]{\mbox {${#1} \rightarrow {#2}$}}
\def\deg[#1,#2]{\mbox {$d_{_{#1}}(#2)$}}
\def\mindeg[#1]{\mbox {$\delta_{_{#1}}$}}
\def\maxdeg[#1]{\mbox {$\Delta_{_{#1}}$}}
\def\outdeg[#1,#2]{\mbox {$d_{_{#1}}^{^+}(#2)$}}
\def\minoutdeg[#1]{\mbox {$\delta_{_{#1}}^{^+}$}}
\def\maxoutdeg[#1]{\mbox {$\Delta_{_{#1}}^{^+}$}}
\def\indeg[#1,#2]{\mbox {$d_{_{#1}}^{^-}(#2)$}}
\def\minindeg[#1]{\mbox {$\delta_{_{#1}}^{^-}$}}
\def\maxindeg[#1]{\mbox {$\Delta_{_{#1}}^{^-}$}}
\def\isdef{\mbox {$\ \stackrel{\rm def}{=} \ $}}
\def\dre[#1,#2,#3]{\mbox {${\cal E}^{^{#3}}(#1,#2)$}}
\def\var[#1,#2]{\mbox {${\rm Var}_{_{#1}}(#2)$}}
\def\ls[#1]{\mbox {$\xi^{^{#1}}$}}
\def\hom[#1,#2]{\mbox {${\rm Hom}({#1},{#2})$}}
\def\onvhom[#1,#2]{\mbox {${\rm Hom^{v}}(#1,#2)$}}
\def\onehom[#1,#2]{\mbox {${\rm Hom^{e}}(#1,#2)$}}
\def\core[#1]{\mbox {$#1^{^{\bullet}}$}}
\def\cay[#1,#2]{\mbox {${\rm Cay}({#1},{#2})$}}
\def\sch[#1,#2,#3]{\mbox {${\rm Sch}({#1},{#2},{#3})$}}
\def\cays[#1,#2]{\mbox {${\rm Cay_{s}}({#1},{#2})$}}
\def\dirc[#1]{\mbox {$\stackrel{\rightarrow}{C}_{_{#1}}$}}
\def\cycl[#1]{\mbox {${\bf Z}_{_{#1}}$}}

\begin{document}
\begin{center}
{\Large \bf On the b-chromatic number of Kneser Graphs}\\
\vspace{0.3 cm}
{\bf Hossein Hajiabolhassan}\\
{\it Department of Mathematical Sciences}\\
{\it Shahid Beheshti University, G.C.}\\
{\it P.O. Box {\rm 1983963113}, Tehran, Iran}\\
{\tt hhaji@sbu.ac.ir}\\ \ \\
\end{center}
\begin{abstract}
\noindent In this note, we prove that for any integer $n\geq 3$
the b-chromatic number of the Kneser graph $KG(m,n)$ is greater
than or equal to $2{\lfloor {m\over 2} \rfloor \choose n}$. This
gives an affirmative answer to a conjecture of \cite{6}.\\

\end{abstract}
\section{Introduction}
The {\it b-chromatic number} of a graph $G$, denoted by
$\chi_b(G)$, is the largest positive integer $t$ such that there
exists a proper coloring for $G$ with $t$ colors in which every
color class contains at least one vertex adjacent to some vertex
in all the other color classes. Such a coloring is called a
b-coloring. Consider a b-coloring for $G$. If $v\in V(G)$ has a
neighbor in any other color class, $v$ is called a {\it
color-dominating vertex} (or simply {\it dominating vertex}). This
concept was introduced in 1999 by Irving and Manlove \cite{5} who
proved that determining $\chi_b(G)$ is $NP$-hard in general and
polynomial for trees. The b-chromatic number of graphs has
received attention recently, see [1-7, 9-11].

Hereafter, the symbol $[m]$ stands for the set $\{0,1,\ldots,
m-1\}$. For a subset $X\subseteq [m]$ denote by ${X \choose n}$
the collection of all $n$-subsets of $X$. Also, the symmetric
difference of the sets $A$ and $B$ is commonly denoted by $A
\Delta B$. Assume that $m \geq 2n$. The {\it Kneser graph}
$KG(m,n)$ is the graph with vertex set ${[m] \choose n}$, in
which $A$ is connected to $B$ if and only if $A \cap B =
\emptyset$. It was conjectured by Kneser \cite{KNE} in 1955, and
proved by Lov\'{a}sz \cite{LOV} in 1978, that
$\chi(KG(m,n))=m-2n+2$.

The b-chromatic number of Kneser graphs has been investigated in
\cite{6}.

\begin{alphcon}\label{omja} {\rm \cite{6}}
For every positive integer $n$, we have
$\chi_b(KG(m,n))=\Theta(m^n)$.
\end{alphcon}

In this regard, it was proved in \cite{6} that
$\chi_b(KG(m,2))=\Theta(m^2)$.

\begin{alphthm}\label{n=2}{\rm \cite{6}}
Let  $m$ be a positive integer and $m\neq 8$. Then
$$\chi_b(KG(m,2))= \left \{ \begin{array}{ll}
\lfloor \frac{m(m-1)}{6} \rfloor& if\ m\ is\ odd\\
 &\\
\lfloor \frac{(m-1)(m-2)}{6}\rfloor+3 & if\ m\ is\ even.
 \end{array}\right. $$
\end{alphthm}

In the next section, we present a lower bound for the b-chromatic
number of Kneser graphs. Next, we show that for any positive
integer $n$, $\chi_b(KG(m,n))=\Theta(m^n)$.
\section{Kneser Graphs}
In what follows we are concerned with some results concerning the
b-chromatic number of Kneser graphs.

\begin{lem}\label{fun}
Let $s$ be a positive integer. If $r \geq 2s+1$, then there
exists a function $g: {[r] \choose s} \longrightarrow {[r]
\choose 2s}$ which satisfies the following conditions.
\begin{enumerate}
\item[{\rm I)}] For any $A\in {[r] \choose s}$, $A\subset g(A)$.
\item[{\rm II)}] For any $A, B\in {[r] \choose s}$, $g(A) \neq g(B)$
whenever $A\cap B=\emptyset$.
\end{enumerate}
\end{lem}
\begin{proof}{For any $A\in {[r] \choose s}$ let $d(A)$ be the smallest positive integer such that $|A
\cup \{\min A+1, \min A+2, \ldots, \min A+d(A) \}|=2s$ where
addition is taken modulo $r$. Now, for any $A\in {[r] \choose s}$
set
$$g(A) \isdef A\cup \{\min A+1, \min A+2, \ldots, \min A+d(A) \}.$$
Note that for any $A, B\in {[r] \choose s}$, if $A\cap
B=\emptyset$, then $\min A \not = \min B$. Hence, it is readily
seen that $g(A) \neq g(B)$ whenever $A\cap B=\emptyset$, as
desired.}
\end{proof}

Here we present a lower bound for the b-chromatic number of the
Kneser graph $KG(m,n)$ provided that $n\geq 3$.

\begin{thm}\label{b}
Let $n\geq 3$ be an integer. If $m \geq 2n$, then
$\chi_b(KG(m,n))\geq 2{\lfloor {m\over 2} \rfloor \choose n}$.
\end{thm}
\begin{proof}{
For $2n \leq m \leq 2n+1$ the assertion follows easily. Hence,
assume that $m\geq 2n+2$. Set $X\isdef \{0,1,\ldots, \lfloor
{m\over 2} \rfloor -1\}$ and $Y\isdef \{\lfloor {m\over 2}
\rfloor,\lfloor {m\over 2} \rfloor+1,\ldots, 2\lfloor {m\over 2}
\rfloor-1\}$. Also, consider a bijective function $f: X
\longrightarrow Y$, e.g., $f(x) \isdef \lfloor {m\over 2}
\rfloor+x$ for any $x\in X$. We show that there exists a
b-coloring $h: V(KG(m,n)) \rightarrow {[m] \choose n}$ which uses
$2{\lfloor {m\over 2} \rfloor \choose n}$ colors. We
consider three cases.\\

Case I) $m$ is even and $n$ is odd ($m=2r$ and $n=2s+1$).\\

In this case we have $X= \{0,1,\ldots,r-1\}$ and $Y=
\{r,r+1,\ldots,2r-1\}$. For any vertex $A \in V(KG(m,n))$ if
$A\subseteq X$ or $A\subseteq Y$, define $h(A)\isdef A$. If $s+1
\leq |A\cap X| \leq 2s$ (resp. $s+1 \leq |A\cap Y| \leq 2s$),
then set $h(A)\isdef B$ where $B$ is an arbitrary $n$-subset of
$Y$ (resp. $X$) such that $f(A\cap X) \cup (A\cap Y) \subseteq B$
(resp. $(A\cap X) \cup f^{-1}(A\cap Y)\subseteq B$). It is
straightforward to check that $h$ is a proper coloring. Clearly,
we use $2{\lfloor {m\over 2} \rfloor \choose n}$ colors. Now, we
show that any vertex $A\in V(KG(m,n))$, where $A\subseteq X$ or
$A\subseteq Y$, is a dominating vertex. Without loss of
generality, we can assume that $A\subseteq X$. Let $B\in
V(KG(m,n))$ where $B\subseteq X$ or $B\subseteq Y$. If $A\cap
B=\emptyset$, then it is clear that $B$ is in the neighborhood of
$A$ and also $h(B)=B$. Now, suppose that $A\cap B\not =\emptyset$
($A\neq B$). It is sufficient to show the existence of a vertex
$C\in V(KG(m,n))$ such that $C$ is adjacent to $A$ and that
$h(C)=B$. To see this, assume that $i\in B\setminus A$. Set
$C\isdef \{i\}\cup f(B \setminus \{i\})$. In view of definition
of $h$,
one can see that $h(C)=B$. Hence, $h$ is a b-coloring.\\

Case II) $m$ and $n$ are even ($m=2r$ and $n=2s$).\\

In this case we have $X= \{0,1,\ldots,r-1\}$ and $Y=
\{r,r+1,\ldots,2r-1\}$. For any vertex $A \in V(KG(m,n))$ if
$A\subseteq X$ or $A\subseteq Y$, define $h(A)\isdef A$. If
$s+1\leq |A\cap X| < 2s$ (resp. $s+1 \leq |A\cap Y| < 2s$), then
set $h(A)\isdef B$ where $B$ is an arbitrary $n$-subset of $Y$
(resp. $X$) such that $f(A\cap X) \cup (A\cap Y) \subseteq B$
(resp. $(A\cap X)\cup f^{-1}(A\cap Y) \subseteq B$).

If $A\cap X =f^{-1}(A\cap Y)$, then set $h(A) \isdef g(A\cap X)$
where $g: {X \choose s} \longrightarrow {X \choose 2s}$ is a
function which satisfies Lemma \ref{fun}. If $A\cap X \neq
f^{-1}(A\cap Y)$, $|A\cap X|=s$ and $\min((A\cap X)\Delta
f^{-1}(A\cap Y)) \in A\cap X$ (resp. $\min((A\cap X)\Delta
f^{-1}(A\cap Y)) \in f^{-1}(A\cap Y)$), then define $h(A)\isdef B$
where $B$ is an arbitrary $n$-subset such that $B\subseteq Y$
(resp. $B\subseteq X$) and $f(A\cap X) \cup (A\cap Y) \subseteq B$
(resp. $(A\cap X) \cup f^{-1}(A\cap Y)\subseteq B$). As in case
I, it is straightforward to check that $h$ is a b-coloring with
$2{\lfloor {m\over 2} \rfloor \choose
n}$ colors.\\

Case III) $m$ is odd ($m=2r+1$).\\

Let $h: V(KG(m-1,n)) \rightarrow {[m-1] \choose n}$ be a
b-coloring for the Kneser graph $KG(m-1,n)$ obtained in the
aforementioned cases. Now, we extend $h$ to a b-coloring $h':
V(KG(m,n))\rightarrow {[m] \choose n}$. Define
$$h'(A)\isdef  \left \{ \begin{array}{ll}
 h(A) & \quad m-1\not\in A\\
 &\\
 \{m-n,m-n+1,\ldots,m-1\}&  \quad m-1\in A.
 \end{array}\right. $$
It is readily seen that $h'$ is a proper coloring for the Kneser
graph $KG(m,n)$. Also, it is straightforward to check that any
vertex $A\in V(KG(m,n))$, where $A\subseteq X$ or $A\subseteq Y$,
is a dominating vertex. Moreover, one can check that the vertex
$\{m-n,m-n+1,\ldots,m-1\}$ is a dominating vertex for the new
color class. Hence, $h'$ is a b-coloring.
}
\end{proof}

It is a simple matter to check that for any graph $G$, $\chi_b(G)
\leq \Delta(G)+1$. Hence, $\chi_b(KG(m,n))=O(m^n)$. In view of
Theorems \ref{n=2} and \ref{b} we have the following corollary.
\begin{cor}
For any positive integer $n$, we have
$\chi_b(KG(m,n))=\Theta(m^n)$.
\end{cor}

\ \\
{\bf Acknowledgement:} The author wishes to thank M. Alishahi and
S. Shaebani for their valuable comments.

\end{document}